\documentclass[12pt]{article}
\usepackage{amsmath}
\usepackage{amssymb}
\usepackage{amsthm}
\usepackage{mathrsfs}
\usepackage{appendix}
\usepackage{hyperref}
\usepackage{geometry}
\usepackage{multirow}
\usepackage{authblk}
\usepackage{tikz}
\usepackage{titlesec}
\usepackage{stmaryrd}

\usetikzlibrary{positioning, shapes.geometric}

\newtheorem{definition}{Definition}
\newtheorem{lemma}{Lemma}
\newtheorem{theorem}{Theorem}
\newtheorem{proposition}{Proposition}
\newtheorem{remark}{Remark}[section]

\newtheorem{assumption}{Assumption}

\newcommand{\e}{\mathrm{e}}
\renewcommand{\i}{\mathrm{i}}
\renewcommand{\d}{\mathrm{d}}

\title{Effective stability for Hamiltonian PDEs with asymptotically vanishing spectral gaps}
\date{}
\author[1]{Bingqi Yu}
\author[,1,2]{Yong Li\thanks{Corresponding author}}
\affil[1]{School of Mathematics, Jilin University, Changchun 130012, People’s Republic of China.}
\affil[2]{Center for Mathematics and Interdisciplinary Sciences,	Northeast Normal University, Changchun 130024, People’s Republic of China.}
\allowdisplaybreaks

\begin{document}
	
	\maketitle
	\footnote{E-mail address: \url{yubq23@mails.jlu.edu.cn}(B. Yu),  \url{liyong@jlu.edu.cn}(Y. Li)}
	
\begin{abstract}
	This paper studies the effective stability of nearly integrable Hamiltonian PDEs with asymptotically vanishing spectral gaps ($0 < \alpha < 1$). Under a unified high-low frequency decomposition, we construct a modified block clustering partition based on Bourgain's ideas. By leveraging the vanishing of spectral gaps to suppress high-frequency resonant contributions, the overall non-resonance property is maintained under high-regularity weights. This framework is applied to space fractional and fully dispersive Whitham-Schr\"odinger equations, uniformly yielding explicit stability estimates in Gevrey, logarithmic ultra-differentiable, and Sobolev spaces.
	
\end{abstract}

\section{Introduction}
The phenomenon of energy transfer in Hamiltonian partial differential equations (PDEs) was regarded by Bourgain as one of the fundamental problems in the field for the 21st century \cite{B00}. Within this research context, two complementary directions have emerged: one explores the significant transfer of energy that leads to the growth of Sobolev norms and associated weak turbulence phenomena \cite{BL25,BLM22,GHP16,GHHMP23,HP15}; the other investigates the localization of energy, namely, whether energy can remain stable across various Fourier modes over long periods. The latter direction primarily encompasses KAM theory for PDEs, which proves the perpetual preservation of invariant tori (see \cite{KP96,B98,BBM14}), and the theory of effective stability, which provides lower bounds on the time scales of stability \cite{BCGW24,BFG20,BMP20,FG13,P18,CLW24,FM23,BFM24,BFLM25,BS13,D09,IP19}. The results of this paper fall into the latter category, belonging to the realm of effective stability, almost global existence, and Nekhoroshev-type stability for PDEs.

In the study of effective stability for Hamiltonian PDEs, the primary mathematical tool is the Birkhoff normal form method. The fundamental idea is that, provided the linear frequencies $\omega_j$ satisfy certain non-resonance conditions, one can eliminate the non-resonant terms in the perturbation by solving the homological equation. During this procedure, the key to establishing small divisor estimates lies in the distribution properties of the frequencies $\omega_j$. Previous results in this area typically rely on frequency separation (or spectral gaps) conditions, such as $|\omega_i-\omega_{j}|>C>0$ for $i\neq j$ (for instance, $\omega_j = j^2$ for the classical NLS, or $\omega_j = |j|$). 

We remark that for the regime of well-separated frequencies where $\omega_{j}\sim|j|^{\alpha},\alpha > 1$ (such as standard NLS or beam equation), we have successfully established almost global existence and Nekhoroshev-type stability across the same broad spectrum of regularities (including Gevrey, log-ultra-differentiable, and other ultra-differentiable class) in our previous work \cite{YL25}. However, when the frequency gaps vanish as the modes increase, the loss of frequency separation renders standard non-resonance conditions invalid. Consequently, research in the critical regime $0 < \alpha < 1$ has been relatively scarce; only \cite{MC20} has considered $\omega_k=(k^2+m)^{\beta}$ with $\beta<\frac{1}{2}$ for polynomial stability times. This paper is devoted to establishing a unified framework to overcome the challenge of these vanishing spectral gaps, yielding sharp sub-exponential and polynomial stability times comparable in regularity breadth to the $\alpha > 1$ case.

\subsection{Main Results}

In this paper, we consider nearly integrable Hamiltonian systems on sequence spaces of the form:
\begin{equation}\label{mainh}
	H(u, \bar{u}) = \sum_{j \in \mathbb{Z}^d} \omega_j(m) |u_j|^2 + P(u, \bar{u}),
\end{equation}
where $m > 0$ is a mass parameter, and the linear frequencies $\omega_j(m)$ satisfy certain growth and asymptotic difference properties (as formulated precisely under Assumption \ref{ass:freq} in Section \ref{sec:2}). Our primary applications focus on two physical models:
\begin{enumerate}
	\item The fully dispersive Whitham-Schr\"odinger (FD-WS) equation:
	\begin{equation}\label{main_whitham}
		\i\partial_t \psi = \sqrt{|D| \tanh(|D|) + m} \, \psi + p(|\psi|^2)\psi, \quad x \in \mathbb{T}^d, t \in \mathbb{R},
	\end{equation}
	where $D = -\i\partial_x$. The positive parameter $m$ acts as a mass or tuning parameter to avoid zero-frequency resonances, which is physically and mathematically analogous to the parameter tuning techniques in Hamiltonian water waves \cite{SHI24}.  Starting from the Euler equations of water waves, while the traditional nonlinear Schr\"odinger (NLS) equation introduces local approximations to the dispersion relation, the Whitham-Schr\"odinger formulation retains the full linear dispersion $\omega_j(m) = \sqrt{|j|\tanh(|j|) + m}$. As $|j| \to \infty$, the spectral gap satisfies $\omega_{j+1} - \omega_j \sim \frac{1}{2\sqrt{|j|}} \to 0$. Recently, substantial progress has been made in investigating the mathematical structures and wave dynamics of Whitham-type and full-dispersion equations, focusing on modulational instabilities, soliton existence, and qualitative behaviors (see, e.g., \cite{SHE18,SHI24}).

	\item The space fractional Schr\"odinger equation (fNLS):
	\begin{equation}\label{main}
		\i\partial_t \psi = (-\Delta + m)^\beta \psi + p(|\psi|^2)\psi, \quad x \in \mathbb{T}^d, t \in \mathbb{R},
	\end{equation}
	where the parameter $\beta$ is in the critical range $0 < \beta < 1/2$, and $p$ is a smooth real-valued function with $p(0)=0$. This model corresponds to linear frequencies $\omega_j(m) = (|j|^2 + m)^\beta$, where the spectral gap $\omega_{j+1} - \omega_j \sim |j|^{2\beta-1} \to 0$.
\end{enumerate}
In recent decades, substantial progress has been made in investigating the dynamical behaviors of the fractional Schr"odinger equation, including its solitary wave solutions, scattering properties, and special perturbation dynamics \cite{G19,GL22,GMPR26,HNS23,IP14,L02,LNZ25,LW24,N20}.

We quantify the regularity of solutions by utilizing a family of weighted Sobolev spaces, which will be formally defined in Section \ref{sec:2}: the Gevrey space $W_{G}^{s,\theta}$, the logarithmic ultra-differentiable space $W_{U}^{s,q}$, and the space of finitely differentiable functions $W_{F}^{s}$.

Throughout this paper, constants with numerical subscripts (e.g., $C_{1}, C_{2}, \dots, C_8$) depend only on the frequency parameter $\alpha$ (or $\beta$), the resonant set measure $\gamma$, and the regularity parameters $s, C_f$; their explicit values are detailed in Appendix A. We now state our main results as follows:

\begin{theorem}[Gevrey class case]
	Let $s > C_1$ and $\theta > 0$. For any interval $[m_1, m_2]$, there exist a set $\mathcal{M} \subset [m_1, m_2]$ of zero Lebesgue measure and a constant $\varepsilon_0 > 0$ such that for any $m \in [m_1, m_2] \setminus \mathcal{M}$, the following holds. If the initial data $\psi_0$ for \eqref{main_whitham} or \eqref{main} satisfies  $\|\psi_0\|_{s,\theta}^G = \varepsilon < \varepsilon_0$, then the corresponding solution $\psi(t)$ satisfies
	\[
	\|\psi(t)\|_{s,\theta}^G \leq C_{2}\varepsilon, \quad \forall |t| \leq C_{3}\exp\left(C_{4}\frac{|\ln \varepsilon|^{\frac{7}{6}}}{\ln|\ln \varepsilon|^{\frac{1}{6}}}\right).
	\]
\end{theorem}

Previous works, such as \cite{BMP20,FG13,P18}, typically established stability times of the form $\exp(C|\ln\varepsilon|^{1+a})$ for some $a>0$. Our result provides a more precise estimate, featuring a logarithmic correction in the exponent.

\begin{theorem}[Logarithmic ultra-differential case]
	Let $s > C_1$ and $q > 0$. For any interval $[m_1, m_2]$, there exist a set $\mathcal{M} \subset [m_1, m_2]$ of zero Lebesgue measure and a constant $\varepsilon_0 > 0$ such that for any $m \in [m_1, m_2] \setminus \mathcal{M}$, the following holds. If the initial data $\psi_0$ for \eqref{main_whitham} or \eqref{main} satisfies  $\|\psi_0\|_{s,q}^U = \varepsilon < \varepsilon_0$, then the corresponding solution $\psi(t)$ satisfies
	\[
	\|\psi(t)\|_{s,q}^U \leq C_{2}\varepsilon, \quad \forall |t| \leq C_{3}\exp\left(C_{4}|\ln\varepsilon|^{\frac{7q}{6q+1}}\right).
	\]
\end{theorem}

Recent works, e.g., \cite{CLW24}, have shown that certain Hamiltonian systems with log-ultra-differentiable regularity exhibit stability over time scales of the form $\exp(C|\ln\varepsilon|^a)$. Our theorem establishes a corresponding result for Hamiltonian PDEs in the critical regime of asymptotically vanishing spectral gaps.

\begin{theorem}[Finite-differentiable case]
	Let $s > C_1$. For any interval $[m_1, m_2]$, there exist a set $\mathcal{M} \subset [m_1, m_2]$ of zero Lebesgue measure and a constant $\varepsilon_0 > 0$ such that for any $m \in [m_1, m_2] \setminus \mathcal{M}$, the following holds. If the initial data $\psi_0$ for \eqref{main_whitham} or \eqref{main} satisfies  $\|\psi_0\|_{s}^F = \varepsilon < \varepsilon_0$, then the corresponding solution $\psi(t)$ satisfies
	\[
	\|\psi(t)\|_{s}^F \leq C_{2}\varepsilon, \quad \forall |t| \leq C_{3}\left(\frac{1}{\varepsilon}\right)^{\frac{2C_{4}}{3}s^{1/7}}.
	\]
\end{theorem}

In \cite{MC20}, the authors also obtained polynomial stability times of the form $(1/\varepsilon)^{a(s)}$ for some function $a(s)$ in the same regime $\beta < 1/2$. Our result provides an explicit expression for the exponent, namely $a(s) = \frac{2C_{4}}{3}s^{1/7}$.

\subsection{Main Techniques and Contributions}

In this paper, we employ a unified framework to establish long-time norm preservation for solutions originating from several function spaces: the Gevrey class ($W^G_{s, \theta}$), the logarithmic ultra-differentiable space ($W^U_{s, q}$), and the space of finitely differentiable functions ($W^F_s$). Our strategy is inspired by the classical Nekhoroshev theorem and is based on a high-low frequency decomposition. Specifically, we construct a normal form for the low-frequency part of the system and show that the resulting error terms can combine with the high-frequency part.

The core challenge of our study stems from the asymptotic properties of the frequencies. Specifically, as $|j| \to \infty$, the spacing between adjacent frequencies vanishes, i.e., $\omega_{j+1} - \omega_j \to 0$. This loss of frequency separation invalidates the standard non-resonance conditions that are the cornerstone of most KAM and Nekhoroshev-type theories.

To address this challenge, we partition the frequency space into blocks $\Omega_{\alpha}$ in order to classify resonant interactions. This partitioning is conceptually inspired by Bourgain's approach on the ``localization of resonant sites'' (see, e.g., \cite{B99, BFM24}), which partitions the lattice $\mathbb{Z}^d = \bigcup_\alpha \Omega_\alpha$ into clusters to handle close-resonance wave couplings \cite{B99}. While the classical Bourgain clustering lemma is based on the quadratic growth of the spectrum (such as the standard Laplacian spectrum with $\omega_j = |j|^2$), the vanishing of spectral gaps in the regime $0 < \alpha < 1$ requires a different adaptation \cite{BFM24}. To satisfy the requirements of norm estimation in high-regularity settings, we modify Bourgain's block clustering partition method to suit our setting of asymptotically vanishing spectral gaps with high-regularity weights. 

Specifically, our approach to handling the vanishing spectral gaps is to utilize the property that they tend to zero: as the spectral gap becomes very small compared to the low-frequency non-zero combinations, the whole denominator are bounded away from zero by the Diophantine conditions on the low modes. Combining this mechanism with the rapid decay of the high-frequency Fourier modes under high regularity, we control the block interactions while recovering a weak non-resonance condition for the block interactions. Crucially, during our normal form iteration, our estimates track the dependence on the truncation scale $N$, allowing the remainder of the iteration to be balanced against the high-frequency tail of the solution. This yields explicit expressions for the stability times.

Furthermore, we highlight that the stability time results achieved under this unified framework represent a significant advancement in terms of mathematical sharpness and generality compared to existing results in the literature. For analytic-like and Gevrey classes, previous works such as \cite{BMP20, FG13, P18} established sub-exponential stability bounds of the form $\exp(C|\ln\varepsilon|^{1+a})$. Our Theorem 1 improves upon these bounds to $\exp(C\frac{|\ln\varepsilon|^2}{\ln|\ln\varepsilon|})$. In the case of logarithmic ultra-differentiable regularity, our Theorem 2 successfully establishes stability times of the form $\exp(C|\ln\varepsilon|^{\frac{7q}{6q+1}})$, extending recent results \cite{CLW24} to the critical regime $0 < \alpha < 1$ with vanishing spectral gaps. Finally, for finitely differentiable (Sobolev) spaces, while earlier investigations in the same asymptotic frequency regime \cite{MC20} yielded polynomial stability times of the generic form $(1/\varepsilon)^{a(s)}$ for some non-explicit exponent function $a(s)$, our Theorem 3 provides a  explicit expression for the exponent, namely $a(s) = \frac{2C_{4}}{3}s^{1/7}$. This comprehensive coverage of regularities within a single robust framework underscores the adaptability of our proposed methodology.

\subsection{Article Structure }

The outline of the paper is as follows. In Section \ref{sec:2}, we introduce the setting of Hamiltonian, function spaces and polynomials on the function spaces. In Section \ref{sec:3}, we define the resonant and nonresonant set to derive the definition of norm form. In Section \ref{sec:4}, we prove lemma to estimate the solution of homological equation and prove the iteration lemma to transform the origin Hamiltonian $H$ to $H\circ\mathcal{T}^{(d)}=H_0+Z_d+R_{d,d}+R_{d,>}.$ In Section \ref{sec:5}, we make terms $R_{d,d}$ and $R_{d,>}$ have the same order, which is expressed by $r$. In Section \ref{sec:6}, we estimate the stability time of the high and low frequency parts respectively to get the conclusion. In Section \ref{sec:7}, we verify the applicability of our framework to the fully dispersive Whitham-Schr\"odinger equation and the fractional NLS equation.
	
	\section{Setting}\label{sec:2}
	
	We denote the index set by $\mathcal{Z} := \mathbb{Z}^\mathsf{d}\times\{-1,1\}$. For an index $J=(j,\sigma)\in\mathcal{Z}$, we define $|J|^2 := |j|^2 = \sum_{l=1}^{\mathsf{d}}|j_l|^2$ and $\langle j\rangle := \max\{1, |j|\}$.
	
	We consider nearly integrable Hamiltonian systems on sequence space of the form
	\begin{equation}\label{Hamiltonian}
		H=H_0+P, \quad \text{with} \quad H_0=\sum_{j\in\mathbb{Z}^\mathsf{d}}\omega_{j}(m)|u_j|^2,
	\end{equation}
	where $m \in [m_1, m_2]$ is a mass parameter. To encompass both the fully dispersive Whitham-Schr\"odinger equation and the fractional NLS equation, we introduce the following assumption on the linear frequencies:
	
	\begin{assumption}\label{ass:freq}
		The linear frequencies $\omega_j(m)$ are smooth in the mass parameter $m \in [m_1, m_2]$. They satisfy the following properties:
		\begin{enumerate}
			\item \textbf{Asymptotic Growth}: For $|j| \to \infty$, they satisfy the following asymptotic behavior:
			$$ \omega_j(m) = |j|^\alpha + r_m(j), \quad \text{with } 0 < \alpha < 1, $$
			where there exists a constant $C_{\mathtt{rem}} > 0$ such that the remainder term $r_m(j)$ satisfies:
			$$ \left| r_m(j) \right| \leq C_{\mathtt{rem}} \langle j \rangle^{\alpha-1}$$
			for any $j \in \mathbb{Z}^{\mathsf{d}}$ and $m \in [m_1, m_2]$.
			\item \textbf{Non-degeneracy}: For any $1 \leq k \leq d$, any compact subinterval $[m_1, m_2]$, and any distinct indices $j_1, j_2, \dots, j_k \in \mathbb{Z}^{\mathsf{d}}$ satisfying $|j_s| < N$ ($s=1,\dots,k$), the determinant of the frequency derivative matrix satisfies:
			$$ \left| \det \left[ \frac{\partial^l \omega_{j_s}(m)}{\partial m^l} \right]_{0 \leq l \leq k-1; 1 \leq s \leq k} \right| \geq \frac{C_{\mathtt{nd}}}{N^{2d^2}} $$
			for some constant $C_{\mathtt{nd}} > 0$ independent of $N$.
		\end{enumerate}
	\end{assumption}
	
	The term $P(u,\bar{u})$ is the representation of the nonlinear interaction in Fourier coordinates, which is assumed to belong to $\mathcal{P}_{4,\infty}$.
	
	Our analysis is conducted in the following family of infinite-dimensional Banach spaces. We expand the wave function $\psi(x,t) \in L^2(\mathbb{T}^{\mathsf{d}})$ via its Fourier series:
	$$ \psi(x,t) = \sum_{j \in \mathbb{Z}^{\mathsf{d}}} u_j(t) \e^{\i j \cdot x}, $$
	and we associate $\psi$ with the sequence $u = (u_J)_{J \in \mathcal{Z}} \in W_s$. For any index $J = (j, \sigma) \in \mathcal{Z}$, the sequence components are defined as $u_{(j,1)} := u_j$ and $u_{(j,-1)} := \bar{u}_j$. Given a weight function $f$, we define the space $W_s$ as
	$$W_s:=\left\{u=(u_J)_{J\in\mathcal{Z}} \mid \|u\|_s < \infty \right\},$$
	endowed with the squared norm
	$$\|\psi\|_s^2 := \| u\|_s^2 = \sum_{j\in\mathbb{Z}^{\mathsf{d}}}|u_j|^2\e^{2sf(\langle j\rangle)}.$$
	\begin{remark}
		For $|j|$ large enough, $\langle j \rangle=|j|$. Consequently, in the subsequent analysis, we will often replace $f(\langle j\rangle)$ with $f(|j|)$ for sufficiently large $|j|$ without loss of generality.
	\end{remark}
	The specific choices of the weight function $f$ define the spaces considered in our main theorems:
	\begin{enumerate}
		\item For $f(x)=x^\theta$ with $0<\theta<1$, we obtain the Gevrey space, denoted by $W^G_{s,\theta}$.
		\item For $f(x)=(\ln x)^q$ with $q>1$, we obtain the logarithmic ultra-differentiable space, denoted by $W^U_{s,q}$.
		\item For $f(x)=\ln x$, the weight becomes $\exp(2s \ln \langle j \rangle) = \langle j \rangle^{2s}$. This yields the standard Sobolev space $H^s$, which we denote by $W_{s}^F$.
	\end{enumerate}
	The analysis presented up to the Normal Form Lemma will be carried out in the abstract weighted space $W_s$.
	
	We denote the open ball in $W_s$ centered at the origin with radius $r$ by $B_s(r)$. A functional $H$ on $W_s$ determines a Hamiltonian vector field $X_H=(X_J)_{J \in \mathcal{Z}}$ whose components are given by
	$$\dot{u}_{J} = X_J(u) := -\sigma \i\frac{\partial H}{\partial u_{\bar{J}}}, \quad \text{where } J=(j,\sigma) \text{ and } \bar{J}=(j,-\sigma).$$
	For a monomial $M=\prod_{l=1}^{d}u_{J_l}$ with multi-index $\mathcal{J}=(J_1,...,J_d)$ where $J_l=(j_l,\sigma_l)$, we define its momentum as $\mathcal{M}(\mathcal{J}) := \sum_{l=1}^{d}\sigma_lj_l.$
	The locality of the nonlinear interaction in \eqref{main_whitham} and \eqref{main} implies that the perturbation term $P$ conserves momentum. This means that $P$ is a sum of monomials whose multi-indices $\mathcal{J}$ satisfy $\mathcal{M}(\mathcal{J})=0$. We denote the set of such momentum-zero multi-indices of degree $d$ by $\mathcal{I}_d$.
	
	For a homogeneous polynomial $P$ of degree $d$, we write
	\begin{equation}\label{P}
		P(u)=\sum_{\mathcal{J} \in \mathcal{Z}^d} P_{\mathcal{J}} u^{\mathcal{J}}, \quad \text{where } \mathcal{J}=(J_1, \dots, J_d) \text{ and } u^{\mathcal{J}} = u_{J_1}\dots u_{J_d}.
	\end{equation}
	\begin{definition}\label{defpol}
		For $d\geq 2$, we denote by $\mathcal{P}_{d}$ the space of formal homogeneous polynomials $P(u)$ of degree $d$ of the form \eqref{P} satisfying:
		\begin{enumerate}
			\item Momentum conservation: $P_{\mathcal{J}} = 0$ if $\mathcal{M}(\mathcal{J}) \neq 0$. That is, $P(u)=\sum_{\mathcal{J}\in\mathcal{I}_d}P_{\mathcal{J}}u^{\mathcal{J}}$.
			\item Reality: $P$ is real-valued, which implies $\overline{P_{\mathcal{J}}}=P_{\overline{\mathcal{J}}}$ for any $\mathcal{J}\in\mathcal{Z}^d$.
			\item Boundedness: The coefficients are uniformly bounded, i.e., $C_P:=\sup_{\mathcal{J}\in\mathcal{I}_d}|P_{\mathcal{J}}|<\infty.$
		\end{enumerate}
	\end{definition}
	For $r,s>0$, we endow the space $\mathcal{P}_d$ with the norm
	$$|P|_{r,s}:=\sup_{u\in B_s(r)}\frac{\Vert X_{P}\Vert_s}{r}.$$
	For integers $\infty> d_2\geq d_1\geq 2$, we let $\mathcal{P}_{d_1,d_2}:=\bigoplus_{k=d_1}^{d_2}\mathcal{P}_k$. Similarly, $\mathcal{P}_{d_1, \infty} := \bigoplus_{k \ge d_1} \mathcal{P}_k$. These spaces are endowed with the same norm topology. The perturbation Hamiltonian $P$ in \eqref{Hamiltonian} belongs to $\mathcal{P}_{4,\infty}$.
	
	For $P_1,P_2\in \mathcal{P}_{d_1,d_2},$ we define their Poisson bracket as
	$$\{P_1,P_2\}:=-\i\sum_{J=(j,\sigma)\in\mathcal{Z}}\sigma\frac{\partial P_1}{\partial u_{J}}\frac{\partial P_2}{\partial u_{\bar{J}}}.$$
	Given a truncation parameter $N > 0$, we decompose the phase space by introducing projectors onto low modes and high modes. For any $u \in W_s$, we define
	$$u^{<} := \Pi^< u, \quad \text{where } (\Pi^< u)_J = \begin{cases} u_J & \text{if } |j| \le N \\ 0 & \text{if } |j| > N \end{cases}$$
	and $u^{>} := \Pi^> u := u - u^<$. This allows us to classify polynomials based on their order of vanishing at $u^>=0$.

	\section{Resonance and Normal Form}\label{sec:3}
	
	Our normal form construction relies on a distinction between resonant and non-resonant interactions. The parameter $m$ is chosen from a set $[m_1,m_2]\setminus\mathcal{M}$ of large Lebesgue measure, which, by Lemma \ref{nonres}, ensures a Diophantine non-resonance condition holds. Specifically, for $N$ sufficiently large, any multi-index $\mathcal{J}=(J_1, \dots, J_d)$ composed entirely of low-frequency modes ($|j_l| \le N$ for all $l$) satisfies
	\begin{equation}\label{nr}
		\left|\sum_{l=1}^{d}\omega_{j_l}\sigma_l\right| \geq \frac{\gamma}{N^{4d^3}}, \quad \text{whenever} \quad \sum_{l=1}^{d}\omega_{j_l}\sigma_l \neq 0.
	\end{equation}
	
	To handle interactions involving high-frequency modes, we partition the index set $\mathcal{Z}$ into disjoint blocks, $\mathcal{Z}=\bigcup_{\alpha}\Omega_{\alpha}$, with the following properties:
	\begin{enumerate}
		\item There is a central block $\Omega_0$ containing all modes up to a certain size: $\Omega_0 := \{ J \in \mathcal{Z} \mid |J| \le C_0^* \}$.
		\item The remaining blocks $\Omega_\alpha$ for $\alpha \neq 0$ have a uniformly bounded "width": for every $\alpha \neq 0$, there exists a constant $C_{1}^*$ such that $\sup_{J\in\Omega_{\alpha}}|J|-\inf_{J\in\Omega_{\alpha}}|J|\leq C_1^*$.
	\end{enumerate}	
	
	Based on these structures, we classify all multi-indices $\mathcal{J}=(J_1,...,J_d)$ according to the number, denoted by $\mathsf{s}$, of its components that are high-frequency (i.e., $|J_i|>N$). The classification, which determines whether an interaction is considered "resonant" or "non-resonant," is summarized in Table \ref{table1}. For interactions with $\mathsf{s}=2$, the distinction also depends on whether the two high-frequency modes belong to the same block $\Omega_\alpha$.
	
	\begin{table}[h]
		\caption{Case classification for resonant and non-resonant multi-indices}
		\label{table1}
		\begin{tabular}{|lll|l|l|}
			\hline
			\multicolumn{3}{|l|}{Classification Condition} & Result & Case Code \\ \hline
			\multicolumn{1}{|l|}{$\mathsf{s}\geq3$} & \multicolumn{2}{l|}{} & High-frequency remainder & $\mathscr{R}$ \\ \hline
			\multicolumn{1}{|l|}{\multirow{3}{*}{$\mathsf{s}=2$}} & \multicolumn{1}{l|}{\multirow{2}{*}{$\sigma_1\sigma_2=-1$}} & $J_1,J_2\in\Omega_{\alpha}$ (same block) & Resonant & R2 \\ \cline{3-5} 
			\multicolumn{1}{|l|}{} & \multicolumn{1}{l|}{} & $J_1\in\Omega_{\alpha},J_2\in\Omega_{\beta},\alpha\neq\beta$ & Non-resonant & NR22 \\ \cline{2-5} 
			\multicolumn{1}{|l|}{} & \multicolumn{2}{l|}{$\sigma_1\sigma_2=1$} & Non-resonant & NR21 \\ \hline
			\multicolumn{1}{|l|}{$\mathsf{s}=1$} & \multicolumn{2}{l|}{ } & Non-resonant & NR1 \\ \hline
			\multicolumn{1}{|l|}{\multirow{2}{*}{$\mathsf{s}=0$}} & \multicolumn{2}{l|}{$\sum_{l=1}^d\omega_{j_l}\sigma_l=0$} & Resonant & R0 \\ \cline{2-5} 
			\multicolumn{1}{|l|}{} & \multicolumn{2}{l|}{$\sum_{l=1}^d\omega_{j_l}\sigma_l\neq0$} & Non-resonant & NR0 \\ \hline
		\end{tabular}
	\end{table}
	
	This classification leads to a partition of all $d$-th order multi-indices into three disjoint sets: the resonant set $\mathscr{Z}$ (cases R0, R2), the non-resonant set $\mathscr{N}$ (cases NR0, NR1, NR21, NR22), and the high-frequency remainder set $\mathscr{R}$ (cases with $\mathsf{s} \ge 3$).
	
	\begin{definition}[Normal Form Polynomial]
		A polynomial $Z \in \mathcal{P}_{3,d}$ is said to be in N-normal form if it is composed exclusively of resonant monomials. That is, if
		$$Z(u)=\sum_{k=3}^{d}\sum_{\mathcal{J}\in\mathcal{I}_k}Z_{\mathcal{J}}u^{\mathcal{J}},$$
		then $Z_{\mathcal{J}}$ can be non-zero only if the multi-index $\mathcal{J}$ belongs to the resonant set $\mathscr{Z}$. Specifically, $\mathcal{J}=\{J_1,\dots,J_k\}$ must satisfy one of the following conditions:
		\begin{enumerate}
			\item Low-frequency resonance, R0 All modes in $\mathcal{J}$ are low-frequency ($|j_l| \le N$ for all $l=1,\dots,k$) and their frequencies are exactly resonant: $\sum_{l=1}^{k}\sigma_l\omega_{j_l}=0$.
			\item High-frequency resonance: The multi-index $\mathcal{J}$ contains exactly two high-frequency modes, say $J_1$ and $J_2$, which form a resonant pair: they have opposite signs ($\sigma_1\sigma_2=-1$) and belong to the same block $\Omega_\alpha$. All other modes in $\mathcal{J}$ are low-frequency.
		\end{enumerate}
	\end{definition}
	
	\section{The Iteration Lemma}\label{sec:4}
	
	We begin by analyzing the solution to the homological equation, which is central to the iterative construction of the normal form.
	
	\begin{lemma}[Homological Equation]\label{homo}
		Let $P \in \mathcal{P}_{k,p}$ be a polynomial of degree at most two with respect to $u^>,$ and let $m \in [m_1,m_2] \setminus \mathcal{M}$. Then there exist a polynomial $G \in \mathcal{P}_{k,p}$ and a normal form polynomial $Z$ such that
		$$P = Z - \{H_0, G\},$$
		where $Z$ is the projection of $P$ onto the resonant monomials, i.e., $Z = \sum_{\mathcal{J}\in\mathscr{Z}} P_\mathcal{J}u^\mathcal{J}$. Moreover, the following estimates hold:
		\begin{enumerate}
			\item $|Z|_{r,s} \leq |P|_{r,s}$;
			\item $|G|_{r,s} \leq \gamma^{-1} N^{C_{5}d^6} |P|_{r,s}$.
		\end{enumerate}
	\end{lemma}
	
	\begin{proof}
		The estimate for $Z$ is straightforward, as its coefficients are a subset of the coefficients of $P$. To construct $G$, we define it as $G = \sum_{\mathcal{J}\in\mathcal{N}} G_{\mathcal{J}}u^{\mathcal{J}}$ and seek to solve the equation
		$-\{H_0,G\} = \sum_{\mathcal{J}\in\mathcal{N}} P_{\mathcal{J}}u^{\mathcal{J}}.$
		This yields the coefficient-wise relation
		$$ \left(\sum_{l=1}^{d}\omega_{j_l}\sigma_l\right) G_{\mathcal{J}} = P_{\mathcal{J}}, \quad \forall \mathcal{J} \in \mathcal{N}. $$
		The primary task is to establish a lower bound for the small divisors, i.e., the term $|\sum_{l=1}^{d}\omega_{j_l}\sigma_l|$, for all non-resonant multi-indices $\mathcal{J} \in \mathcal{N}$. We proceed by analyzing each case from Table \ref{table1}.
		
		For Case NR0, all indices satisfy $|J_l|\leq N$ and $\sum_{l}\sigma_l\omega_{j_l}\neq0$. The non-resonance condition \eqref{nr} directly implies that $\left|\sum_{l=1}^{d}\omega_{j_l}\sigma_l\right|\geq\frac{\gamma}{N^{4 d^3}}$.
		
		For Case NR1, we have the inequalities
		$$\left|\sum_{l\neq1}\sigma_l\omega_{j_l}\right|\leq (d-1)(N^\alpha + C_{\mathtt{rem}} N^{\alpha-1}) \leq c_2(d-1)N^\alpha \quad \text{and} \quad |\omega_{j_1}|\geq |j_1|^\alpha - C_{\mathtt{rem}} |j_1|^{\alpha-1} \geq c_1 |j_1|^\alpha$$
		for $|j_1| > N$. Consequently, for $|j_1|\geq \left(\frac{c_2}{c_1}\right)^{\frac{1}{\alpha}} (2d)^{\frac{1}{\alpha}}(N+m_2):=N_1$, we have
		$$\left|\sum_{l=1}^{d}\omega_{j_l}\sigma_l\right|\geq c_1 |j_1|^{\alpha}-c_2(d-1)N^\alpha\geq N^\alpha > 1.$$
		For the regime where $|j_l|\leq|j_1|\leq N_1$ for $l\neq1,$ applying \eqref{nr} with an effective scale of $N_1$ yields
		$$\left|\sum_{l=1}^{d}\omega_{j_l}\sigma_l\right|\geq\frac{\gamma}{N_1^{4 d^3}}=\frac{\gamma}{(\left(\frac{c_2}{c_1}\right)^{\frac{1}{\alpha}}(2d)^{\frac{1}{\alpha}}(N+m_2))^{4d^3}}\geq\frac{\gamma}{N^{4d^3+\frac{8}{\alpha}d^4}}.$$
		
		For Case NR21, we have the following inequalities:
		$$\left|\sum_{l\neq1,2}\sigma_l\omega_{j_l}\right|\leq c_2(d-2)N^\alpha \quad \text{and} \quad |\omega_{j_1}|+|\omega_{j_2}|\geq c_1|\max\{|j_1|,|j_2|\}|^{\alpha}.$$ 
		When $\max\{|j_1|,|j_2|\}\geq \left(\frac{c_2}{c_1}\right)^{\frac{1}{\alpha}} (2d)^{\frac{1}{\alpha}}(N+m_2)=N_1$, it follows that
		$$\left|\sum_{l=1}^{d}\omega_{j_l}\sigma_l\right|\geq c_1 |\max\{|j_1|,|j_2|\}|^\alpha - c_2(d-2)N^\alpha \geq 2N^\alpha > 1.$$
		When $|j_l|\leq\max\{|j_1|,|j_2|\}\leq N_1$ for $l\neq1,2$, we again have
		$$\left|\sum_{l=1}^{d}\omega_{j_l}\sigma_l\right|\geq\frac{\gamma}{N_1^{4 d^3}}\geq\frac{\gamma}{N^{4d^3+\frac{8}{\alpha}d^4}}.$$
		
		For Case NR22, we invoke \eqref{nr} to bound the low-frequency part of the sum
		$$\left|\sum_{l\neq1,2}\sigma_l\omega_{j_l}\right|\geq \frac{\gamma}{N^{4(d-2)^3}}.$$
		Applying Assumption \ref{ass:freq} to the high-frequency part gives
		\begin{align*}
			|\omega_{j_1}-\omega_{j_2}| &\leq \left| |j_1|^{\alpha}-|j_2|^{\alpha} \right| + \left| r_m(j_1) - r_m(j_2) \right| \\
			&\leq \alpha |j_2|^{\alpha-1} |j_1 - j_2| +2 C_{\mathtt{rem}}  |j_2|^{\alpha-1}.
		\end{align*}
		By momentum conservation, $|j_1 - j_2| \leq (d-2)N$. Hence:
		$$ |\omega_{j_1}-\omega_{j_2}| \leq \left( \alpha(d-2)N +2 C_{\mathtt{rem}}  \right) |j_2|^{\alpha-1}. $$
		When $|j_2|\geq N^{\frac{4d^3}{1-\alpha}},$ we obtain
		$$\left|\sum_{l=1}^d\sigma_l\omega_{j_l}\right|\geq\frac{\gamma}{N^{4(d-2)^3}}-\frac{(\alpha(d-2)N + 2C_{\mathtt{rem}})}{N^{4d^3}}\geq\frac{\gamma}{2N^{4(d-2)^3}}\geq\frac{\gamma}{N^{4d^3}}.$$
		And when $|j_2|\leq N^{\frac{4d^3}{1-\alpha}}, |j_1|\leq(d-2)N+N^{\frac{4d^3}{1-\alpha}}\leq N^{\frac{5d^3}{1-\alpha}},$ applying \eqref{nr} on the scale of $N_2 = N^{\frac{5d^3}{1-\alpha}}$ yields
		$$\left|\sum_{l=1}^{d}\omega_{j_l}\sigma_l\right|\geq\frac{\gamma}{(N^{\frac{5d^3}{1-\alpha}})^{4d^3}}=\frac{\gamma}{N^{\frac{20d^6}{1-\alpha}}}.$$ 
		
		Combining the estimates from all cases, we conclude that the small divisors for all non-resonant indices satisfy the uniform bound:
		$$ \left|\sum_{l=1}^{d}\omega_{j_l}\sigma_l\right|\geq\frac{\gamma}{N^{C_{5}d^6}},$$
		where $C_5 = \max\left\{ \frac{20}{1-\alpha}, 1 + \frac{8}{\alpha} \right\}$. This lower bound on the divisors directly leads to the estimate for the norm of the generator $G$:
		$$|G|_{r,s}\leq \gamma^{-1} N^{C_{5}d^6}|P|_{r,s}.$$
	\end{proof}

	Now we can begin the iterate process.

	\begin{lemma}[The Iteration Lemma]\label{iter}
		Let $H$ be the Hamiltonian defined in \eqref{Hamiltonian}, and let the parameters satisfy $d \geq k \geq 3$, $d > C_7$, $s > s_0 > \frac{\mathsf{d}}{2}$, and the smallness condition $r d^2 C_{6} N^{C_{5}d^6} < 1$. Define a sequence of shrinking radii by $r_k := 2r - \frac{(k-3)r}{d-3}$. Then there exists a sequence of canonical transformations $\mathcal{T}^{(k)} : B_s(r_k) \to B_s(r_3)$ such that the transformed Hamiltonian $H^{(k)} := H \circ \mathcal{T}^{(k)}$ possesses the following properties:
		\begin{enumerate}
			\item It admits the decomposition $H^{(k)} = H_0 + Z_k + P_k + R_{k,d} + R_{k,>}$;
			\item The perturbation term $P_k \in \mathcal{P}_{k,d}$ is bounded by $|P_k|_{r_k,s} \leq d^{2k-6}(C_{P,1}r)^{k-2}N^{C_{5}(k-3)d^6}$;
			\item The high-order remainder $R_{k,d}$ is bounded by $|R_{k,d}|_{r_k,s} \leq r^{d-2}N^{3C_{5}d^{7}}$;
			\item The high-frequency remainder $R_{d,>}$ is bounded by $|R_{d,>}|_{r,s} \leq \frac{C_{R}r}{\e^{(s-s_0)f(N)}}$.
		\end{enumerate}  
	\end{lemma}

	\begin{proof}
		The proof proceeds by induction on the index $k$.
		
		For induction base $k=3$, we begin with the initial Hamiltonian $H = H_0+P$. For a given large integer $N$, we decompose the perturbation $P$ as $P = P_3 + R_{3,>} + R_{3,d}$, where $R_{3,d} \in \mathcal{P}_{d+1,\infty}$ contains terms of order $d+1$ and higher, and $R_{3,>}$ contains terms with a zero of at least order three in $u^>$. This provides the base case for our iteration: $H^{(3)} := H_0+P_3+R_{3,>}+R_{3,d}$.

		Now we consider following homological equations in the inductive step  $d\geq k\geq3$:
		$$\{H_0,G_{k+1}\}+P_k=Z_{k}^*,$$  
		where $Z_3:=Z_3^*.$ 
		The next transformation is the time-$1$ map, $\Phi_{G_{k+1}}$, generated by the Hamiltonian flow of $G_{k+1}$. We define the subsequent Hamiltonian as
		\begin{align*}
			H^{(k+1)} := H^{(k)} \circ \Phi_{G_{k+1}} = (H_0+Z_{k}+P_{k}+R_{k,d}+R_{k,>}) \circ \Phi_{G_{k+1}}.
		\end{align*}

		Consider the transformation acting on term $H_0$. Let $n$ be an integer such that $n=[\frac{k-d}{k-2}]$, then
		\begin{align*}
			H_0\circ\Phi_{G_{k+1}} =H_0+\{H_0,G_{k+1}\}+\sum_{l=2}^{n}\frac{ad^l_{G_{k+1}}}{l!}H_0+R_{H_0,G_{k+1}}\\
			&=H_0+\{H_0,G_{k+1}\}+P_{H_0,k+1}+R_{H_0,G_{k+1}}.
		\end{align*}
		By the choice of $n$, we have $P_{H_0,k+1}\in\mathcal{P}_{k+1,d},R_{H_0,G_{k+1}}\in\mathcal{P}_{d+1,\infty}$.

		Consider the transformation acting on term $Z_k$:
		\begin{align*}
			Z_k\circ\Phi_{G_{k+1}}&=Z_k+\sum_{l=1}^{n}\frac{ad^l_{G_{k+1}}}{l!}Z_k+R_{Z_k,G_{k+1}}\\
			&=Z_k+P_{Z_k,k+1}+R_{Z_k,G_{k+1}},
		\end{align*}
		where $n$ is taken the same way as above, so there is $P_{Z_k,k+1}\in\mathcal{P}_{k+1,d},R_{Z_k,G_{k+1}}\in\mathcal{P}_{d+1,\infty}$.
		
		Consider the transformation acting on term $P_k$:
		\begin{align*}
			P_k\circ\Phi_{G_{k+1}}&=P_k+\sum_{l=1}^{n}\frac{ad^l_{G_{k+1}}}{l!}P_k+R_{P_k,G_{k+1}}\\
			&=P_k+P_{P_k,k+1}+R_{P_k,G_{k+1}}.
		\end{align*}
		There is also $P_{P_k,k+1}\in\mathcal{P}_{k+1,d},R_{P_k,G_{k+1}}\in\mathcal{P}_{d+1,\infty}$.
		
		Consider the transformation acting on term $R_{k,d}$. Note that $R_{k,d}\circ\Phi_{G_{k+1}}$ has at least $d+1$ order zero at $u=0$. Hence
		$$R_{k,d}\circ\Phi_{G_{k+1}}\in\mathcal{P}_{d+1,\infty}.$$
		
		Consider the transformation acting on term $R_{k,>}$. Now
		$R_{k,>}\circ\Phi_{G_{k+1}}$ has at least a third order zero at $u^>=0$.
		
		Summing up the above, we can rearrange as follows for the $H^{(k+1)}$:
		\begin{align*}
			H^{(k+1)}:&=H_0+(\{H_0,G_{k+1}\}+P_k)+Z_k+\\
			&+P_{H_0,k+1}+P_{Z_k,k+1}+P_{P_k,k+1}\\
			&+R_{H_0,G_{k+1}}+R_{Z_k,G_{k+1}}+R_{P_k,G_{k+1}}+R_{k,\bar{p}}\circ\Phi_{G_{k+1}}\\
			&+R_{k,>}\circ\Phi_{G_{k+1}}.
		\end{align*} 
		Define
		\begin{align*}
			R_{d,k+1}:&=R_{H_0,G_{k+1}}+R_{Z_k,G_{k+1}}+R_{P_k,G_{k+1}}+R_{k,\bar{p}}\circ\Phi_{G_{k+1}},\\
			P_{k+1}^*:&=P_{H_0,k+1}+P_{Z_k,k+1}+P_{P_k,k+1}.
		\end{align*}
		We decompose $P_{k+1}^*$ into $P_{k+1}^*=P_{k+1}+ R_{k+1,>}^*$, where $R_{k+1,>}^*$ contains all terms with a zero of order at least three in $u^>$, and $P_{k+1} \in \mathcal{P}_{k+1,d}$ contains terms where the order of the zero in $u^>$ is at most two. Substituting these definitions and using the homological equation, we arrive at the desired structure:
		\begin{align*}
			H^{(k+1)}&=H_0+( Z_{k}^*)+Z_k+P_{k+1}^*+R_{k+1,d}+R_{k,>}\circ\Phi_{G_{k+1}}\\
			&=H_0+(Z_k+ Z_{k}^*)+P_{k+1}+R_{k+1,d}+(R_{k,>}\circ\Phi_{G_{k+1}}+ R_{k+1,>}^*)\\
			&:=H_0+Z_{k+1}+P_{k+1}+R_{k+1,d}+R_{k+1,>}. 
		\end{align*}
		We now establish the bounds stated in the lemma. Lemma \ref{homo} provides the initial estimates for the generator and the new normal form term:
		\begin{align*}
			|G_{k+1}|_{r_{k},s}&\leq\frac{|P_k|_{r_k,s}}{\gamma}N^{C_{5}d^6},\\
			|Z_{k+1}-Z_k|_{r_k,s}&\leq| Z_{k}^*|_{r_k,s}\leq|P_k|_{r_k,s}.
		\end{align*}
		
		We use induction to prove the estimates of $P_k$ and $G_{k+1}$ during the iteration process.
		
		By the choice of $r$, we have the first inductive step
		\begin{align*}
			|G_4|_{r_3,s}&\leq\frac{|P_3|_{r_3,s}}{\gamma}N^{C_{5}d^6}\leq\frac{2C_Pr}{\gamma}N^{C_{5}d^6}\\
			&\leq E:=1/{16\e d}<\frac{r_k-r_{k+1}}{8\e r_k}.
		\end{align*} 
		Then we can use Lemma \ref{Lie} to prove the estimates for $P_{k+1},G_{k+2}$ based on the estimates for $P_{k},G_{k+1}$ inductively. Note that
		\begin{align*}
			|P_{k+1}|_{r_{k+1},s}&\leq|P_{H_0,k+1}|_{r_{k+1},s}+|P_{Z_k,k+1}|_{r_{k+1},s}+|P_{P_k,k+1}|_{r_{k+1},s}\\
			&\leq|\sum_{l=2}^{n}\frac{ad^l_{G_{k+1}}}{l!}H_0|_{r_{k+1},s}+|\sum_{l=1}^{n}\frac{ad^l_{G_{k+1}}}{l!}Z_k|_{r_{k+1},s}+|\sum_{l=1}^{n}\frac{ad^l_{G_{k+1}}}{l!}P_k|_{r_{k+1},s}\\
			&\leq|\sum_{l=1}^{n-1}\frac{ad_{G_{k+1}}^l}{(l+1)!}\{G_{k+1},H_0\}|_{r_{k+1},s}+|\sum_{l=1}^{n}\frac{ad^l_{G_{k+1}}}{l!}(Z_k+P_k)|_{r_{k+1},s}\\
			&\leq\sum_{l=1}^{n-1}\frac{1}{(l+1)!}(\frac{|G_{k+1}|_{r_k,s}}{2E})^l|P_k|_{r_k,s}\\
			&+\sum_{l=1}^{n}\frac{1}{l!}(\frac{|G_{k+1}|_{r_k,s}}{2E})^l(|P_k|_{r_k,s}+\sum_{m=3}^{k-1}|Z_{m+1}-Z_m|'_{r_k,s}+|Z_3|'_{r_k,s})\\
			&\leq\sum_{l=1}^{n}\frac{2}{l!}(\frac{|G_{k+1}|_{r_k,s}}{2E})^l(\sum_{m=3}^{k}|P_m|'_{r_m,s})\\
			&\leq \frac{\e|G_{k+1}|_{r_k,s}}{E}\sum_{m=3}^{k}|P_m|'_{r_m,s}.
		\end{align*}
		When $k=3$, we have 
		$$|P_4|_{r_4,s}\leq \frac{\e}{E\gamma } C_P^2r_3^2N^{C_{5}d^6}\leq\frac{64\e^2d}{\gamma }C_P^2r^2N^{C_{5}d^6}.$$
		Therefore, when $k\geq4$, we will use induction to prove that there exists a constant $C_{P,1}$ such that $|P_{k}|'_{r_k,s}\leq d^{2k-6}(C_{P,1}r)^{k-2}N^{C_{5}(k-3)d^6} $ holds for $k\geq 4$. And 
		\begin{align*}
			|P_{k+1}|_{r_{k+1},s}&\leq\frac{\e}{\gamma E}d^{2k-6}(C_{P,1}r)^{k-2}N^{C_{5}(k-2)d^6}(\sum_{l=3}^{k}d^{2l-6}(C_{P,1}r)^{l-2}N^{C_{5}(l-3)d^6})\\
			&\leq \frac{16\e^2}{\gamma}d^{2k-5}(C_{P,1}r)^{k-2}N^{C_{5}(k-2)d^6}\frac{C_{P,1}r}{1-d^2C_{P,1}rN^{C_{P,1}d^6}}\\
			&\leq \frac{32\e^2}{\gamma}d^{2k-5}(C_{P,1}r)^{k-1}N^{C_{5}(k-2)d^6}\\
			&\leq d^{2k-4}(C_{P,1}r)^{k-1}N^{C_{5}(k-2)d^6}.
		\end{align*}
		Here we use the setting of $r,d$. Then 
		\begin{align*}
			|G_{k+2}|_{r_{k+1},s}&\leq\frac{1}{\gamma} d^{2k-4}(C_{P,1}r)^{k-1}N^{C_5(k-1)d^6}\\
			&\leq \frac{1}{d^2\gamma}(C_{P,1}rd^2N^{C_{5}d^6})^{k-1}\leq E
		\end{align*}
		by Lemma \ref{homo} and the setting of $r$.  
		Thus we have completed the inductive proof of the estimates for $P_k,G_{k+1}.$

		It follows from the definition of norm that $\sup_{u\in B^s(r_{k+1})}\Vert X_{G_{k+1}}\Vert_{s}\leq r_{k+1}|G_{k+1}|_{r_{k+1},s},$ which leads to the near-identity property of $\Phi_{G_{k+1}}$:
		$$\sup_{u\in B_s(r_{k+1})}\Vert (\Phi_{G_{k+1}}-Id)\circ(u)\Vert_s\leq\int_{0}^{1}\sup_{u\in B_s(r_{k+1})} \Vert X_{G_{k+1}}(u(T))\Vert_sdT\leq r_{k+1}E\leq r_{k}-r_{k+1} .$$
		Namely the transformation maps $B_s(r_{k+1})$ into $B_s(r_k)$.
		
		Besides, from the integral-type remainder  $$R_{X,G_{k+1}}=\frac{1}{n!}\int_{0}^{1}(1-T)^n(ad_{G_{k+1}}^{n+1}X)\circ\Phi_{G_{k+1}}^T\d T,X={G_{k+1},H_0},Z_k,P_k,$$
		we get 
		\begin{align*}
			|R_{X,G_k+1}|_{r_k,s}&\leq\frac{1}{n!}|X|_{r_k,s}(\frac{|G_{k+1}|}{2E})^n\\
			&\leq d^{2k-6}(C_{P,1}r)^{k-2}N^{C_{5}(k-3)d^p}(\frac{8\e d}{\gamma} d^{2k-6}(C_{P,1}r)^{k-2}N^{C_{5}(k-2)d^p})^n  \\
			&\leq  d^{2nk+2k-5n-6}(C_{P,1}r)^{(k-2)(n+1)}N^{C_{5}(n(k-2)+k-3)d^p}(\frac{8\e}{\gamma})^n           \\
			&\leq d^{2d-4} (C_{P,1}r)^{d-2}N^{C_{5}(d+k-5)d^p}(\frac{8e}{\gamma d})^n\\
			&\leq (C_{P,1}r)^{d-2}N^{2C_{5}d^7}
		\end{align*}
		by Lemma \ref{Lie}. Then by the iterative process involving $R_{k+1,d}$,
		\begin{align*}
			|R_{k+1,d}|_{r_{k+1},s}&\leq |R_{H_0,G_k+1}|_{r_k,s}+|R_{Z_k,G_k+1}|_{r_k,s}+|R_{P_k,G_k+1}|_{r_k,s}+|R_{k,d}\circ\Phi_{G_{k+1}}|_{r_k,s}\\
			&\leq 3(C_{P,1}r)^{d-2}N^{2C_{5}d^7}+(1+E)|R_{k,d}|_{r_k,s},\\
			\frac{|R_{k+1,d}|_{r_{k+1},s}}{(1+E)^{k+1}}&\leq\frac{3}{(1+E)^{k+1}} (C_{P,1}r)^{d-2}N^{2C_{5}d^7}+\frac{|R_{k,d}|_{r_k,s}}{(1+E)^k},\\
			|R_{k,d}|_{r_k,s}&\leq3\frac{(1+E)^{k-3}-1}{E}(C_{P,1}r)^{d-2}N^{2C_{5}d^7}+(1+E)^{k-3}C_Pr^{d-2}\\
			&\leq48\e d(\e^{\frac{1}{16\e}}-1)(C_{P,1}r)^{d-2}N^{2C_{5}d^7}+eC_Pr^{d-2}\\
			&\leq  r^{d-2}   N^{3C_{5}d^{7}}.
		\end{align*} 
		We thus derive the estimate of $R_{k,d}$.
		
		Since $R_{k,>}\in\mathcal{P}_{3,\infty}$, we can use Lemma \ref{cut} and the choice of $r$ to derive
		$|R_{d,>}|_{r_k,s}\leq\frac{rC_{R}}{\e^{(s-s_0)f(N)}}.$
		
		Finally, the transformation $\mathcal{T}^{(k)}=\Phi_{G_{4}}\circ...\circ\Phi_{G_{k+1}}:B_s(r_k)\to B_s(r_{3})$ is the desired transformation.
		
	\end{proof}

	\section{Normal Form Lemma}\label{sec:5}
	In this section, we balance the order of the two remaining terms in the Iteration Lemma to obtain the Normal Form Lemma.
	\begin{theorem}[Normal Form Lemma]\label{Normal}
		For $H$ as defined in (\ref{Hamiltonian}), let $d\gg3,P\in\mathcal{P}_{3,\infty}.$ Then there exist $N_{d}$ and a canonical transformation $\mathcal{T}_{d}$ such that for $s>s_0,$ 
		the following holds for any sufficiently small $r$:
		\begin{align*}
			&\mathcal{T}_{d}:B_s(r)\to B_s(2r),\\
			&\mathcal{T}_{d}^{-1}:B_s(2r)\to B_s(r),\\
			&H^{(d)}:=H\circ\mathcal{T}_d=H_0+Z_d+R_d,
		\end{align*}	 
		where
		\begin{enumerate}
			\item $Z_d\in\mathcal{P}_{3,d}$ is in the $N-$cutting normal form;
			\item $|R_d|_{r,s}\leq \e^{-C_{4}f(N(r))}$.
			
		\end{enumerate}
		
		The relationship between $N(r)$ and $r$ is implicitly defined in the proof.

		Moreover, $Z_d$ can be decomposed as $Z_d=Z_0+Z_{>}$, such that the index in $Z_0$ is in the case R0 and the index in $Z_>$ is in the case NR2, 
		and $$\sup_{u\in B_s(r)}\Vert(Id-\Pi^>)(X_{Z_>})\Vert_{s}\leq \e^{-C_{\mathtt{fin}}f(N(r))}. $$
		
	\end{theorem}
	
	\begin{proof}
		First, we set $k=d$ in Iteration Lemma \ref{iter}, the Hamiltonian becomes  $H'=H_0+Z_{d}+R_{d,d}+R_{d,>}$ with $$|R_{d,d}|_{r_k,s}\leq r^{d-2}N^{3C_{5}d^{7}},|R_{d,>}|\leq C_R\frac{r}{\e^{(s-s_0)f(N)}}.$$ 
		Next, we will adjust the parameters in the estimate of $R_{d,d}$ and $R_{d,>}$ to make them of the same order of magnitude, which allows them to be combined into a single remainder term. When $d>C_7,s>C_1,rC_R<1,d'=\frac{d}{C_{4}}$, the above estimate simplifies
		\begin{align*}
			|R_{d,d}|_{r_k,s}&\leq r^{d-2}N^{3C_{5}d^{7}}
			\leq (r^{d'}N^{d'^{7}})^{C_{4}}.
			\end{align*}
				Meanwhile,
				$$|R_{d,>}|\leq \e^{-C_{4}f(N)}.$$ 
				for the sake of simplicity, we continue to denote $r',d'$ by $r,d$. Now we impose the condition 
				$$(r(N)^{d^6})^d=\e^{-f(N)},$$
				namely
				\begin{equation}\label{order}
					d^6\ln(N)+\ln r=\frac{-f(N)}{d}.
				\end{equation}
				
				To equate the orders of magnitude, we set
				$$d^6\ln(N)=\frac{f(N)}{d}.$$
				Then specifying $f(x)$, we can derive the dependence of $N$ on $d$ and substitute it back to \eqref{order} to determine its dependence on $r$. Therefore the condition of Lemma \ref{iter} reduces to the condition that $r$ must be sufficiently small.

				Finally, by the cutting lemma and the definition of $Z_y$, the term $Z_>$ is of the same order of magnitude as $R_d$.

			\end{proof}

	We now analyze the order of remainder term in the Normal Form Lemma for parameter $r$, focusing on two representative cases of the function $f(x)$ that represent infinitely differentiable regularities. These cases are formalized in the following propositions.

		\begin{proposition}
				When $f(x)=x^\theta,$ with $\theta<1,\e^{f(N)}$ has same order to $\exp(\frac{|\ln r|^{\frac{7}{6}}}{(\ln|\ln r|)^{\frac{1}{6}}})$. \\
				
			\end{proposition}
			\begin{proof}
				When $f(x)=x^\theta,$ we proceed with the following calculations
				\begin{align*}
					d^{6}(\ln N)&=\frac{N^\theta}{d}=\frac{|\ln r|}{2}, \\
					\frac{\ln N^{\theta}}{\theta}&=\frac{N^{\theta}}{d^7},\\ 
					\ln\frac{\theta N^{\theta}}{d^7}&=\frac{\theta N^\theta}{d^7}+\ln\frac{\theta}{d^7},\\
					-\frac{\theta N^\theta}{d^7}\e^{-\frac{\theta N^\theta}{d^7}}&=-\frac{\theta}{d^7},\\
					W_{-1}(-\frac{\theta}{d^7})&=-\frac{\theta N^\theta}{d^7},\\
					N^\theta&=-\frac{d^7}{\theta}W_{-1}(-\frac{\theta}{d^7}).
				\end{align*}	
				To compare the order of stability time expressed in $r$, we have
				\begin{align*}
					|\ln r|&=-\frac{d^6}{\theta}W_{-1}(-\frac{\theta}{d^7}),\\
					\ln|\ln r|&=6\ln d-\ln\theta+\ln(-W_{-1}(-\frac{\theta}{d^7})) .
					\end{align*}
						Then we get
						\begin{align*}
							\frac{N^\theta(\ln|\ln r|)^{\frac{1}{6}}}{|\ln r|^{\frac{7}{6}}}&=(\frac{6\ln d-\ln\theta+\ln(-W_{-1}(-\frac{\theta}{d^7}))}{-\frac{1}{\theta}W_{-1}(-\frac{\theta}{d^7})})^{\frac{1}{6}}\\
							&=\left(\frac{\frac{6}{7}\ln\frac{\theta}{d^7}+\frac{\ln \theta}{7}-\ln(-W_{-1}(-\frac{\theta}{d^7}))}{\frac{1}{\theta}W_{-1}(-\frac{\theta}{d^7})}\right)^{\frac{1}{6}}.
						\end{align*}
						We claim that $\lim\limits_{y\to0^-}\frac{\ln(-y)}{W_{-1}(y)}=1$, the conclusion then follows from the limit:
						$$\lim\limits_{d\to+\infty}\frac{N^\theta(\ln|\ln r|)^{\frac{1}{6}}}{|\ln r|^{\frac{7}{6}}}=(\frac{6\theta}{7})^{\frac{1}{6}}.$$
						To prove the claim, we calculate the limit by substituting $y=xe^x$:
						
						$$\lim\limits_{y\to0^-}\frac{\ln(-y)}{W_{-1}(y)}=\lim\limits_{x\to-\infty}=\frac{\ln(-x)+x}{x}=1.$$
					\end{proof}

					\begin{proposition}
						When $f(x)=(\ln (x))^q,$ with $q>1,\e^{f(N)}$ is of the same order as $|\ln r|^{\frac{7q}{6q+1}}$.
					\end{proposition}
					\begin{proof}
						When $f(x)=(\ln (x))^q,$ we proceed with the following calculations  
						\begin{align*}
							d^6\ln(N)&=\frac{\ln(N)^q}{d},\\
							d^7&=\ln(N)^{q-1},\\
							(\ln N)^q&=d^{\frac{7q}{q-1}},\\
							|\ln r|&=d^{\frac{6q+1}{q-1}},\\
							(\ln N)^q&=|\ln r|^{\frac{7q}{6q+1}}.
						\end{align*}
					\end{proof}
					
					For the case of finite differentiability, where $f(x) = \ln(x)$, the remainder is addressed in the following proposition.
					
					\begin{proposition}
						When $f(x)=\ln(x),s=d^7$, the remain terms $|R_{d,>}|,|R_{d,d}|\leq \e^{-C_{4}f(N)}\leq  r^{\frac{2C_{4}}{3}s^{\frac{1}{7}}}.$ 
					\end{proposition}
					\begin{proof}
						Notice that 
						\begin{align*}
							(rN^{d^6})^d&=\e^{-2s\ln(N)}=N^{-2s},\\
							N&=(\frac{1}{r})^{\frac{d}{d^7+2s}},\\
							N^{-2s}&=r^{\frac{2ds}{2s+d^7}}.
						\end{align*}
						Since $$\frac{2ds}{2s+d^7}=\frac{1}{\frac{1}{d}+\frac{d^6}{2s}}\leq\frac{\sqrt[7]{2s}}{2},$$ 
						$d=\sqrt[7]{2s}$ is the optimal choice of $d$, it yields $N^{-2s}=r^{\frac{2d}{3}}=r^{\frac{2}{3}{s}^{\frac{1}{7}}}$.
					\end{proof}

					\section{Stability Time}\label{sec:6}
					To apply the Normal Form Lemma, we set $w=\mathcal{T}^{(d)}(u),w_0=\mathcal{T}^{(d)}(u_0),$ and we consider the Cauchy problem
					\begin{equation}\label{Cau}
						\dot{w}=X_{H^{(d)}}(w),w(0)=w_0.
					\end{equation}
					Let $z(t)$ be the solution of \eqref{Cau} and define
					$$T_r:=sup\{|t|\in\mathbb{R}^+\mid\Vert w\Vert_s\leq C_{2}r\}$$
					as the escape time of the solution from the ball of radius $R$. Next, we split the normal form as $Z_d=Z_0+Z_>$, as stated in Theorem \ref{Normal}. We obtain the following system of equations:
					\begin{align}
						\dot{w}^<&=\Lambda w^{<}+X_{Z_0}(w^<)+\Pi^<X_{Z_>}(w^<,w^>)+\Pi^{<}X_{R_d}(w^<,w^>),\label{low}  \\
						\dot{w}^>&=\Lambda w^>+\Pi^>X_{Z_>}(w^<,w^>)+\Pi^>X_{R_d}(w^<,w^>).\label{high}
					\end{align}
					We first give a standard priori estimate on the low frequency part $w^<$ of the solution of \eqref{Cau} based on \eqref{low}. 
					\begin{proposition}
						For $s>C_1$, and any real $w_0$ with $\Vert w_0\Vert_{s}<r $ in \eqref{Cau}, we have
						$$\Vert w^<(t)\Vert_s\leq\Vert w^>(0)\Vert_{s}+r\e^{-C_{4}f(N(r))}|t|, \forall|t|\leq T_r.$$
					\end{proposition}
					\begin{proof}
						Since $\{|w_{(j,+)}|^2,Z_0\}=0,$ we have
						\begin{align*}
							\frac{\d}{\d t
							}\Vert w^<\Vert_s^2&=\{\Vert w^<\Vert_s^2,Z_2\}+\{\Vert w^<\Vert_s^2,R_d\}\\
							&=\sum_{J\in\mathcal{Z}}\frac{\partial}{\partial u_{J}}(\Vert w^<\Vert_s^2)\cdot(X_{Z_>}(w^<,w^>)+X_{R_d}(w^<,w^>))\\
							&\leq r(|Z_>|'_{2r,s}+|R_d|'_{2r,s})=r\e^{-C_{4}f(N(r))}.
						\end{align*}
						The last inequality follows from the definition of the norm $|\cdot|'_{r,s}.$
					\end{proof}
					We now proceed to estimate the high-frequency modes.
					\begin{proposition}
						For $s>s_0$ and any real $w_0$ with $\Vert w_0\Vert_{s}<r\leq C_{6}$ in \eqref{Cau}, we have
						$$\Vert w^>(t)\Vert_s\leq C_{8}(\Vert w^>(0)\Vert_{s}+r\e^{-C_{4}f(N(r))}|t |), \forall|t|\leq T_r.$$
					\end{proposition}
					\begin{proof}
						First, we define a family of linear operators $\mathcal{L}(w^<):\Pi^>W_s\to\Pi^>W_s$  such that $X_{Z_>}(w^<,w^>)=\mathcal{L}(w^<)w^>$, and denote $\mathcal{L}(t):=\mathcal{L}(w^<(t)).$ 
						
						Then for any $w^>\in\Pi^>W_s,$ we introduce the projectors defined as follows:
						$$\Pi_\alpha:\Pi^>W_s\to\Pi^>W_s\ ,  (w_{(j,\sigma)})_{(j,\sigma)}\mapsto (w_{(j,\sigma)} \chi_{\Omega_{\alpha}}(j))_{(j,\sigma)},$$ 
						where $\chi_{\Omega_{\alpha}}$ is the indicator function on $\Omega_{\alpha}.$
						Then we can split $w$ as follows:
						$$\forall w\in\Pi^>W_s,\ w=\sum_{\alpha}w_\alpha,\ w_\alpha:=\Pi_{\alpha}w.$$
						Similarly, by the definition of case NR2, $\mathcal{L}(t)$ has a block-diagonal structure and can be written as
						$$\mathcal{L}(t)=\sum_{\alpha}\mathcal{L}_{\alpha}(t),\ \mathcal{L}_\alpha(t)=\Pi_\alpha\mathcal{L}(t)\Pi_{\alpha}.$$
						For any block $\Omega_{\alpha}$ we define 
						$|\alpha|=\inf_{j\in\Omega_{\alpha}}|j|$. 
						We now consider the part of \eqref{high} corresponding to the normal form:
						\begin{equation}\label{normeq}
							\partial_t w_{\alpha}(t)=\Lambda w_\alpha+\mathcal{L}_\alpha(t)z_{\alpha}(t).
						\end{equation}
						Since $\mathcal{L}_\alpha$ is Hamiltonian, we have
						$$\Vert w_\alpha(t)\Vert_{\ell^2}=\Vert w_\alpha(t_0)\Vert_{\ell^2}, \forall t,t_0\in [-T_r,T_r],$$
						therefore, $\forall |t|\leq T_r$		    
						\begin{align*}
							\Vert w(t)\Vert_s&=\sum_{\alpha}\sum_{j\in\Omega_{\alpha}}\e^{2sf(\langle j\rangle)}|w_{(j,\sigma)}(t)|^2 \\
							&\leq \sum_{\alpha}\sum_{j\in\Omega_{\alpha}}\e^{2sf(|\alpha|+C_1^*)}|w_{(j,\sigma)}(t)|^2 \\
							&\leq \sum_{\alpha}\e^{2s(f(|\alpha|)+C_f f(C_1^*))}\Vert w_{\alpha}(t)\Vert_{\ell_2}^2 \\
							&={C_{8}} \sum_{\alpha}\e^{2sf(|\alpha|)}\Vert w_{\alpha}(0)\Vert_{\ell_2}^2 \\
							&\leq {C_{8}}\sum_{\alpha}\sum_{j\in\Omega_{\alpha}}\e^{2sf(|j|)}|w_{(j,\sigma)}(0)|^2 \\
							&={C_{8}}\Vert w(0)\Vert_{s}.
						\end{align*}
						Hence, if we denote the flow map of \eqref{normeq} by $\mathcal{W}(t,\tau)$, we have: 
						$$\Vert\mathcal{W}(t,\tau)w_0\Vert_s\leq{C_{8}}\Vert w(0)\Vert_{s}.$$
						The solution to \eqref{high} is therefore given by
						$$w^>(t)=\mathcal{W}(t,0)w_0+\int_{0}^{t}\mathcal{W}(t,\tau)\Pi^>X_{R_d}(w^<,w^>)d\tau.$$
						Thus
						$$\Vert w^>(t)\Vert_s\leq{C_{8}}\Vert w_0\Vert_{s}+r{C_{8}}\e^{-C_{4}f(N)}|t|.$$

					\end{proof}
					
					We now combine the estimates for the low and high modes to get
					$$\Vert\psi(t) \Vert_s\leq (C_{8}+1)\Vert\psi_0\Vert_s+r(C_{8}+1)\e^{-C_{4}f(N)}|t|.$$
					Take $T^*=C_{3}\e^{C_{4}f(N)}.$ When $|t|<T^*,$
					\begin{align*}
						\Vert\psi(t)\Vert_s&\leq(C_{8}+1)r+(C_{8}+1)C_{3}r\\
						&\leq C_{2}r.
					\end{align*}
					Then, taking the norms $\Vert\cdot\Vert_s $ as the specific norms in three specific spaces, the main result can be obtained based on the discussion of the remainder size according to the Normal Form Lemma.

					\section{Applications to Physical Models}\label{sec:7}
					
					In this section, we apply our abstract results to two physically significant systems, demonstrating that both the fully dispersive Whitham-Schr\"odinger equation and the space fractional NLS equation fit into our general framework.
					
					\subsection{Fully Dispersive Whitham-Schr\"odinger Equation}
					We consider the fully dispersive Whitham-Schr\"odinger (FD-WS) equation:
					$$ \i\partial_t \psi = \sqrt{|D| \tanh(|D|) + m} \, \psi + p(|\psi|^2)\psi, \quad x \in \mathbb{T}^{\mathsf{d}}, t \in \mathbb{R}. $$
					The linear frequencies of this system correspond to
					$$ \omega_j(m) = \sqrt{|j|\tanh(|j|) + m}. $$
					To show that these frequencies satisfy Assumption \ref{ass:freq} with $\alpha = 1/2$, we write $ \omega_j(m) = |j|^{1/2} + r_m(j), $
					where the remainder term is given by:
					$$ r_m(j) = \sqrt{|j|\tanh(|j|) + m} - |j|^{1/2}. $$
					For $|j| \geq 1$, we can express the remainder as:
					$$ r_m(j) = |j|^{1/2} \left( \sqrt{\tanh(|j|) + \frac{m}{|j|}} - 1 \right). $$
					Using the well-known asymptotic expansion of the hyperbolic tangent, $\tanh(|j|) = 1 - 2\e^{-2|j|} + \mathcal{O}(\e^{-4|j|})$ as $|j| \to \infty$, we have:
					\begin{align*}
						r_m(j) &= |j|^{1/2} \left( \left( 1 + \frac{m}{|j|} - 2\e^{-2|j|} + \mathcal{O}(\e^{-4|j|}) \right)^{1/2} - 1 \right) \\
						&= |j|^{1/2} \left( \frac{m}{2|j|} - \e^{-2|j|} + \mathcal{O}\left( \frac{m^2}{|j|^2} \right) \right) \\
						&= \frac{m}{2|j|^{1/2}} - |j|^{1/2}\e^{-2|j|} + \mathcal{O}\left( |j|^{-3/2} \right).
					\end{align*}
					Thus, there exists a constant $C_{\mathtt{rem}} > 0$ such that for any $j \in \mathbb{Z}^{\mathsf{d}}$ and $m \in [m_1, m_2]$,
					$$ \left| r_m(j) \right| \leq C_{\mathtt{rem}} \langle j \rangle^{-1/2}, $$
					which satisfies the decay condition under Assumption \ref{ass:freq} with $\alpha = 1/2$.
					
					Furthermore, for any $1 \leq k \leq d$ and distinct indices $|j_1| < \dots < |j_k| < N$, we verify the non-degeneracy condition. The derivatives of $\omega_j(m)$ with respect to $m$ are given by
					$$ \frac{\d^l \omega_{j}}{\d m^l} = c_l (|j|\tanh(|j|) + m)^{\frac{1}{2} - l}, $$
					where $c_l = \prod_{n=0}^{l-1}(1/2 - n)$ for $l \ge 1$ and $c_0 = 1$. Consider the determinant
					$$
					D = \det \left[ \frac{\partial^l \omega_{j_s}(m)}{\partial m^l} \right]_{0 \leq l \leq k-1; 1 \leq s \leq k} = \det \left[ c_l (|j_s|\tanh(|j_s|) + m)^{\frac{1}{2} - l} \right].
					$$
					Factoring out the constants and the square root terms, we obtain
					$$
					D = \left( \prod_{l=0}^{k-1} c_l \right) \left( \prod_{s=1}^k (|j_s|\tanh(|j_s|) + m)^{1/2} \right) \det \left[ x_s^l \right]_{0 \leq l \leq k-1; 1 \leq s \leq k},
					$$
					where $x_s = (|j_s|\tanh(|j_s|) + m)^{-1}$. The latter is a Vandermonde determinant which can be explicitly evaluated as:
					$$
					\det \left[ x_s^l \right] = \prod_{1 \leq r < s \leq k} (x_s - x_r) = \prod_{1 \leq r < s \leq k} \frac{|j_r|\tanh(|j_r|) - |j_s|\tanh(|j_s|)}{(|j_r|\tanh(|j_r|) + m)(|j_s|\tanh(|j_s|) + m)}.
					$$
					Since $|j_s| \leq N$ and $|j_r|\tanh(|j_r|) - |j_s|\tanh(|j_s|) \geq C > 0$ for distinct integers, we find that the determinant is bounded from below by:
					$$ |D| \geq \frac{C_{\mathtt{nd}}}{N^{2d^2}}. $$
					Therefore, we have successfully verified both the asymptotic growth and the non-degeneracy conditions under Assumption \ref{ass:freq} for the fully dispersive Whitham-Schr\"odinger equation. Consequently, the main results (Theorems 1, 2, and 3) hold for this model.
					
					\subsection{Space Fractional Schr\"odinger Equation}
					We consider the space fractional Schr\"odinger (fNLS) equation:
					$$ \i\partial_t \psi = (-\Delta + m)^\beta \psi + p(|\psi|^2)\psi, \quad x \in \mathbb{T}^{\mathsf{d}}, t \in \mathbb{R} $$
					with critical range $0 < \beta < 1/2$. The linear frequencies of this system correspond to
					$$ \omega_j(m) = (|j|^2 + m)^\beta. $$
					By setting $\alpha = 2\beta \in (0,1)$, we can express the frequencies as:
					$$ \omega_j(m) = |j|^\alpha + r_m(j), $$
					where the remainder term is given by:
					$$ r_m(j) = (|j|^2 + m)^{\alpha/2} - |j|^\alpha. $$
					For $|j| \geq 1$, we can express the remainder as:
					$$ r_m(j) = |j|^\alpha \left( \left( 1 + \frac{m}{|j|^2} \right)^{\alpha/2} - 1 \right). $$
					Using the binomial expansion for $|j| \to \infty$, we obtain:
					\begin{align*}
						r_m(j) &= |j|^\alpha \left( \frac{\alpha m}{2|j|^2} + \mathcal{O}\left( \frac{m^2}{|j|^4} \right) \right) \\
						&= \frac{\alpha m}{2} |j|^{\alpha-2} + \mathcal{O}\left( |j|^{\alpha-4} \right).
					\end{align*}
					Since $\alpha - 2 < \alpha - 1$, there exists a constant $C_{\mathtt{rem}} > 0$ such that for any $j \in \mathbb{Z}^{\mathsf{d}}$ and $m \in [m_1, m_2]$,
					$$ \left| r_m(j) \right| \leq C_{\mathtt{rem}} \langle j \rangle^{\alpha-1}, $$
					which satisfies the decay condition under Assumption \ref{ass:freq}.
					
					Furthermore, for any $1 \leq k \leq d$ and distinct indices $|j_1| < \dots < |j_k| < N$, we verify the non-degeneracy condition. The derivatives of $\omega_j(m)$ with respect to $m$ are given by
					$$ \frac{\d^l \omega_{j}}{\d m^l} = c_l (|j|^2 + m)^{\beta - l}, $$
					where $c_l = \prod_{n=0}^{l-1} (\beta-n)$ for $l \ge 1$ and $c_0 = 1$. Consider the determinant
					$$
					D = \det \left[ \frac{\partial^l \omega_{j_s}(m)}{\partial m^l} \right]_{0 \leq l \leq k-1; 1 \leq s \leq k} = \det \left[ c_l (|j_s|^2 + m)^{\beta-l} \right].
					$$
					Factoring out the constants and the exponent terms, we obtain
					$$
					D = \left( \prod_{l=0}^{k-1} c_l \right) \left( \prod_{s=1}^k (|j_s|^2 + m)^\beta \right) \det \left[ x_s^l \right]_{0 \leq l \leq k-1; 1 \leq s \leq k},
					$$
					where $x_s = (|j_s|^2 + m)^{-1}$. The latter is a Vandermonde determinant which can be explicitly evaluated as:
					$$
					\det \left[ x_s^l \right] = \prod_{1 \leq r < s \leq k} (x_s - x_r) = \prod_{1 \leq r < s \leq k} \frac{|j_r|^2 - |j_s|^2}{(|j_r|^2 + m)(|j_s|^2 + m)}.
					$$
					Since $|j_s| \leq N$ and $|j_r|^2 - |j_s|^2 \geq 1$ for distinct integers, while $|j_s|^2 + m \le 2N^2$, we find that the determinant is bounded from below by:
					$$ |D| \geq \frac{C_{\mathtt{nd}}}{N^{2d^2}}. $$
					Therefore, we have successfully verified both the asymptotic growth and the non-degeneracy conditions under Assumption \ref{ass:freq} for the space fractional Schr\"odinger equation. Consequently, the main results (Theorems 1, 2, and 3) hold for this model.

					\appendix
					\renewcommand{\thesection}{\Alph{section}}
					\section*{Appendices}
					\section{Constants}
					\begin{align*}
						C_1&=s_0+C_{4},\\
						C_{2}&=2C_{8}.\\
						C_{3}&=\frac{C_{8}-1}{C_{8}+1}
						,\\
						C_{4}&=\frac{1}{2\sqrt[6]{6C_{5}}},\\
						C_{5}&=\max\{\frac{20}{1-\alpha},1+\frac{8}{\alpha} \},\\
						C_{P,1}&=\frac{8\e C_P}{\sqrt{\gamma}},\\
						C_{6}&=\max\{\frac{32C_P\e}{\gamma},2,\frac{24\e^2}{\gamma},\frac{16C_{P,1}\e}{\gamma},\log_3{48\e},\e^{2sC_ff(C_1^*)}\},\\
						C_7&=\frac{32\e^2}{\gamma},\\
						C_{8}&=\e^{-2sC_ff(C_1^*)}.
					\end{align*}

					\section{Technical Lemmas }\label{sec:B}
					
					\begin{lemma}\label{weight}
						The functions $f=x^\theta$ and $f=(\ln(x))^q,q\geq1$ both satisfy 
						$$f(\sum_{l=1}^{d}x_l)\leq f(x_1)+C_f\sum_{l=2}^{d}f(x_l),$$
						where $x_1\geq\dots\geq x_d\geq c, C_f<1.$
					\end{lemma}  
					\begin{proof}
						For $f(x)=x^\theta$, we take $c=2,C_f=2^{\theta-1}$, and prove the case of two variables:
						$$(x_1+x_2)^\theta\leq x_1^\theta+2^{\theta-1}x_2^\theta,x_1\geq2,x_2\geq2.$$ 
						By making the homogenizing substitution $t=\frac{x_2}{x_1}\leq1$, it suffices to prove:
						$$(1+t)^\theta\leq1+\frac{(2t)^\theta}{2},0<t\leq1,$$
						which is readily verified.
						Then for $x_d < ... < x_1$, we can get
						$$f(\sum_{l=1}^{d}x_l)\leq f(\sum_{l=1}^{d-1}x_l)+2^{\theta-1}f(x_d)\leq... \leq f(x_1)+C_f\sum_{l=2}^{d}f(x_l).$$
						
						For $f(x)=(\ln x)^q,$ we take $F(x)=f(x+x_2)-f(x)$. Notice that 
						$$f''=q(\ln x)^{q-2}(\frac{q-1-\ln(x)}{x}),$$ so when $c\geq\e^q, f''<0$, $F'(x)=f'(x+x_2)-f'(x)<0$, we thus have $F(x_1)\leq F(x_2)=f(2x_2)-f(x_2)$. Hence we just need to prove:
						$$f(2x_2)-f(x_2)\leq C_f f(x_2),$$
						namely $$\frac{\ln (x_2)+\ln2}{\ln(x_2)}\leq(1+\frac{\ln 2}{\ln c})^{\frac{1}{q}},$$
						which holds for sufficiently large $c$. The conclusion then follows by induction.
					\end{proof}
					
					\begin{lemma}[Norm estimate for $P$]\label{norm}
						When $s>s_0$, for any $P\in\mathcal{P}_d,d\geq3,$ we have $$|P|_{r,s}\leq C_Pr^{d-2},$$ where $s_0$ satisfies $\sum_{J\in\mathcal{Z}}\e^{(2C_f-2)s_0f(\langle J\rangle)}<\frac{1}{3}.$ 
					\end{lemma}
					\begin{proof}
						Let $P=\sum_{\mathcal{J}\in\mathcal{I}_d}P_{\mathcal{J}}u_{J_1}...u_{J_d}.$ We denote the multi-index $\{J_1,...J_{k-1},(j,-1),J_{k+1},...,J_d\}$ by $\hat{J}_{k,j}$, and  $\{J_1,...J_{k-1},J_{k+1},...,J_d\}$ by $\hat{J}_k.$ Then
						\begin{align*}
							(X_P)_{j,+1}&=-\i\sum_{k=1}^{d}\sum_{\hat{J}_{k,j}\in \mathcal{I}_d}P_{\hat{J}_{k,j}}u_{\hat{J}_k},\\
							|(X_P)_{j,+1}|&\leq C_P\sum_{k=1}^{d}\sum_{\hat{J}_{k,j}\in \mathcal{I}_d}|u_{\hat{J}_k}|,\\
							|(X_P)_{j,+1}|\e^{sf(\langle j\rangle)}&\leq C_P\sum_{k=1}^{d}\sum_{\hat{J}_{k,j}\in \mathcal{I}_d}|u_{\hat{J}_k}|\e^{sf(|j|)}.
						\end{align*}
						Notice that $|j|=|\mathcal{M}(J_1,...,J_{k-1},J_{k+1},J_d)|,$ from $\mathcal{M}(\hat{J}_{k,j})=0.$   When $d\geq3$,  we have
			\begin{align*}
			sf(\langle j\rangle)\leq sf(\sum_{l\neq k}\langle J_l\rangle)\leq sf(\langle J_m\rangle)+sC_f(\sum_{l\neq m,k}\langle J_l\rangle).
							\end{align*} 
	We omit a technical discussion here. 
Then 
	\begin{align*}
	|(X_P)_{(j,+1)}|e^{sf(\langle j\rangle)}&\leq C_P\sum_{k=1}^{d}\sum_{\hat{J}_{k,j}\in \mathcal{I}_d}		\e^{(1-C_f)sf(\langle J_m\rangle)}\prod_{J\in\hat{J}_k}|u_{J}|e^{sC_f(\langle J\rangle)},\\
	\sum_{j\in\mathbb{Z}}|(X_P)_{(j,+1)}\e^{sf(\langle j\rangle)}|&\leq C_P^2
		\sum_{j\in \mathbb{Z}} 
    \left(\sum_{k=1}^{d}\sum_{\hat{J}_{k,j}\in \mathcal{I}_d}
	\e^{(1-C_f)sf(\langle J_m\rangle)}\prod_{J\in\hat{J}_k}|u_{J}|\e^{sC_f(\langle J\rangle)}\right)^2\\
	&\leq d^2C_P^2(\sum_{J_m}|u_m|\e^{sf(\langle J_m\rangle)})\prod_{J\neq J_m,J\in\hat{J}_k}\left(\sum_{J}|u_J|\e^{sC_ff(\langle J\rangle )} \right)^2\\
	&\leq d^2C_P^2(\sum_{J_m}|u_m|\e^{2sf(\langle J_m\rangle)})\\
	&\prod_{J\neq J_m,J\in\hat{J}_k}(\sum_{J}|u_J|^2\e^{2sf(\langle J\rangle)})(\sum_{J}\e^{(2C_f-2)s_0f(|J|)})^{d-2}\\
	&\leq \frac{d^2}{9^{d-2}}C_P^2\Vert u\Vert_s^{2d-2}.
    \end{align*}
	Thus, the estimate $\Vert X_P\Vert_s\leq C_P\Vert u\Vert_s^{d-1}$ implies the conclusion $|P|_{r,s}\leq C_Pr^{d-2}$.
	\end{proof}
								
	\begin{lemma}[Truncating Lemma]\label{cut}
	For $s>s_0,$
	if a polynomial $P\in\mathcal{P}_d, d\geq3$ is of degree at least $3$ with respect to $u^>=0$, then we have
	$$|P|_{r,s} \leq C_{P}\frac{(2r)^{d-2}}{e^{(s-s_0)f(N)}}.$$ 
									
\end{lemma}
	\begin{proof}
	First, we expand the polynomial $P$ as $P(u)=\sum_{l=3}^{d}P'_{J,l}(u^>)^l(u^<)^{d-l}$. Then following and the proof of Lemma 3.8 in \cite{BFM24} , we have
									\begin{align*}
										\Vert(X_P)(u^>,u^<)\Vert_s 
										&\leq C_{P}2^d(\Vert u^>\Vert_{s_0}\Vert u^<\Vert_{s}^{d-2}+\Vert u^>\Vert_{s_0}^2\Vert u^<\Vert_s^{d-3}),
									\end{align*}
									and
									$$\Vert u^>\Vert_{s_0}^2=\sum_{|J|>N}e^{2s_0 f(|J|)}|u_{J}|^2=\sum_{|J|>N}\frac{e^{2sf(|J|)}|u_J|^2}{e^{2(s-s_0)f(|J|)}}\leq\frac{\Vert u\Vert_s^2}{e^{2(s-s_0)f(|N|)}}.$$
									So
									$$\sup_{u\in B^s(r)}\Vert(X_P)(u^>,u^<)\Vert_s\leq C_{P}\frac{2^dr^{d-1}}{e^{(s-s_0)N}}, $$
									which means
									$$|P|_{r,s}\leq C_{P}\frac{2^dr^{d-2}}{e^{(s-s_0)f(N)}}.$$
								\end{proof}

								\begin{lemma}[Lie bracket estimate]\label{Lie}
									Given two polynomials $P\in\mathcal{P}_{p},Q\in\mathcal{P}_{q},|Q|_{r,s}\leq\delta:=\frac{\rho}{8e(r+\rho)},$ we have $\{P,Q\}\in\mathcal{P}_{p+q-2}$ and $ |\{P,Q\}|_{r,s}\leq |P|_{r+\rho,s}|Q|_{r+\rho,s}\frac{1}{2\delta}$. Moreover, 
									$$|ad_{Q}^kP|_{r,s}\leq|P|_{r+\rho,s}(\frac{|Q|_{r+\rho,s}}{2\delta})^k.$$
								\end{lemma}
								The proof can be found in Appendix B in \cite{BMP20}.

								\begin{lemma}\label{det}
									Let $u^{(1)},...,u^{(k)}$ be $k$ independent vectors with $\Vert u^{(i)} \Vert_{\ell^1}\leq1.$ Let $w\in\mathbb{R}^k$ be an arbitrary vector, then there exists $i\in\{1,...,k\},$ such that 
									$$|u^{(i)}\cdot w|\geq\frac{\Vert w\Vert_{\ell^1}\text{det}(u^{(1)},...,u^{(k)})}{k^{\frac{3}{2}}}.$$
								\end{lemma}
								\begin{lemma}\label{measure}
									Suppose that $g(\tau)$ is $m$ times differentiable on an interval $J\subset\mathbb{R}.$ Let $J_{h}:=\{m\in J\mid |g(\tau)|<h \},h>0.$ If on $J,|g^{(m)}(\tau)|\geq d>0,$ then $|J_h|\leq2(2+3+...+m+d^{-1})h^{\frac{1}{m}}.$ Here, $|\cdot| $ denotes the Lebesgue measure. 
								\end{lemma}
								
								Lemma \ref{det} and Lemma \ref{measure} are from Lemma 5.2 and Lemma 5.4 in  \cite{BG06}.

								\begin{lemma}\label{measure2}
									For frequency satisfies Non-degeneracy in Assumption \ref{ass:freq},
									the following non-resonant set 
									$$\mathcal{M}_\gamma=\{m\in[m_1,m_2]\mid\sum_{l=1}^{d}\sigma_l\omega_{j_l}\leq\frac{\gamma}{N^{4d^3}}, \exists \{j_1,...,j_d\}, |j_l|<N\}$$
									has measure estimate:
									$$|\mathcal{M}_\gamma|\leq \gamma.$$ 
								\end{lemma} 
								\begin{proof}
									By Lemma \ref{det}, for any $\mathcal{J}=\{j_1,...,j_d\}$ satisfying $|j_l|<N,$ we can get an index $(i)$ such that
									$$|\sum_{l=1}^{d}\frac{\d^{(i)}\omega_{j_l}(m)}{\d m^{(i)}}|\geq\frac{C_{\beta}d}{N^{2d^2+2}}.$$
									We fix $\mathcal{J}$ to define $$\mathcal{M}_{\mathcal{J,\gamma}}:=\{m\in[m_1,m_2]\mid\sum_{l=1}^{d}\sigma_l\omega_{j_l}(m)\leq\frac{\gamma}{N^{4d^3}}\}.$$
									Then by Lemma \ref{measure}, we have
									\begin{align*}
										|\mathcal{M}_{\mathcal{J},\gamma}|&\leq\left(\frac{\gamma}{N^{4d^3}}\right)^{\frac{1}{(i)}}\frac{N^{2d^2+2}}{C_\beta}\\
										&\leq( \frac{\gamma^{\frac{1}{(i)}}}{N^{4d^2}} )\frac{N^{2d^2+2}}{C_\beta}\leq \frac{C_\beta\gamma}{N^{d^2}}.
										\end{align*}
											Thus
											\begin{align*}
												|\mathcal{M}_\gamma|&\leq\sum_{\mathcal{J},|j_l|<N}|\mathcal{M}_{\mathcal{J},\gamma}|\\
												&\leq\sum_{\mathcal{J},|j_l|<N}\frac{C_\beta\gamma}{N^{d^2}}\leq\frac{C_\beta\gamma(2N)^d}{N^{d^2}}\leq\gamma.
											\end{align*} 
										\end{proof}

	\begin{lemma}\label{nonres}
	There exists a set $\mathcal{M}$ of measure zero such that for any
	$m\in[m_1,m_2]\setminus\mathcal{M}$, 
	$N$ large enough, $\forall J_1,\dots,J_d $ with $|J_l|\leq N,l=1,\dots,d$, and $\sum_{l=1}^{d}\omega_{j_l}\sigma_l\neq0$, we have 
	$$\left|\sum_{l=1}^{d}\omega_{j_l}\sigma_l\right|\geq\frac{\gamma}{N^{4d^3}}.$$
	\end{lemma}
	\begin{proof}
	The set
	$$\mathcal{M}:=\bigcap_{\gamma>0}\mathcal{M}_\gamma$$ is the desired set. $\mathcal{M}_\gamma$ is defined in Lemma  \ref{measure2}. Then
    $$|M|\leq|\mathcal{M}_\gamma|\leq\gamma,\forall\gamma>0.$$
	Moreover, $\forall m\in[m_1,m_2]\setminus\mathcal{M},\exists\gamma>0$ satisfies the non-resonant condition.
	\end{proof}

	\section*{Data Availability Statement}
	Data sharing is not applicable to this article as no new data were created or analyzed in this study.
										
	\section*{Acknowledgement}
	The second author (Y. Li) was supported by National Natural Science Foundation of China (12071175, 12471183 and 12531009).

\end{document}